\documentclass{amsart}

\usepackage{color}

\usepackage[colorlinks=true,
linkcolor=blue,
anchorcolor=blue,
citecolor=blue
]{hyperref}%so that we can include $$ in the title

\usepackage{amsmath,amssymb,amsthm,amsfonts,mathrsfs}
\newtheorem{theorem}{Theorem}[section]

\newtheorem{corollary}[theorem]{Corollary}

\theoremstyle{definition}

\theoremstyle{remark}
\newtheorem{remark}[theorem]{Remark}

\newcommand{\Q}{\mathbb{Q}}

\newcommand{\cO}{\mathcal{O}}

\newcommand{\fkp}{\mathfrak{p}}

\renewcommand{\leq}{\leqslant}
\renewcommand{\geq}{\geqslant}

\newcommand{\of}[1]{\Big(#1\Big)}

\newcommand{\set}[1]{\left\{#1\right\}}

\newcommand{\smat}[1]{\begin{smallmatrix}#1\end{smallmatrix}}

\newcommand{\on}{\operatorname}

\newcommand{\Gal}{\on{Gal}}

\title[The Ramanujan sum and Chebotarev densities]{The Ramanujan sum and Chebotarev densities}
\author[B. Wang]{Biao Wang}
\address{Department of Mathematics, State University of New York at Buffalo, Buffalo, NY 14260-2900, USA}
\email{bwang32@buffalo.edu}
\date{}
\makeatletter
\@namedef{subjclassname@2020}{\textup{}2020 Mathematics Subject Classification}
\makeatother
\begin{document}

\begin{abstract}
		In this short note, we show an analogue of one of Alladi's and Dawsey's formulas with respect to the Ramanujan sum $c_n(m)$ for $m\geq1$. Their formulas may be viewed as the case $m=1$ in our result.
\end{abstract}

\maketitle

\section{Introduction and statement of results}
Let $\mu(n)$ be the M\"{o}bius function defined by $\mu(1)=1$, $\mu(n)=(-1)^k$ if $n$ is the product of $k$ distinct primes and is zero otherwise. It is well-known (e.g.,  \cite{d82}) that the prime number theorem is equivalent to the assertion that
$
\sum_{n=1}^\infty\frac{\mu(n)}{n}=0
$, 
or equivalently,
\begin{equation}\label{pnt}
	-\sum_{n=2}^\infty\frac{\mu(n)}{n}=1.
\end{equation}

Let $k\geq 1$ be an integer.  In 1977, Alladi \cite{a77} refined  (\ref{pnt}) to a formula on primes in arithmetic progressions. Explicitly, he showed that for any integer $\ell$ with $(\ell,k)=1$, we have
\begin{equation}\label{alladi}
	-\sum_{\smat{n\geq 2\\ p(n)\equiv \ell (\on{mod}k)}}\frac{\mu(n)}{n}=\frac1{\varphi(k)},
\end{equation}
where $p(n)$ is the smallest prime divisor of $n$ and $\varphi$ is the Euler totient function. 

In 2017, Dawsey \cite{d17} first generalized Alladi's formula (\ref{alladi}) to the setting of Chebotarev densities for finite Galois extensions
of $\Q$.  In 2019, Sweeting and Woo \cite{sw19}  generalized   (\ref{alladi}) further to number fields and Kural, McDonald and Sah \cite{kms19} generalized all of these results to the general densities of sets of primes. In \cite{w19}, we showed an analogue of Alladi's and Dawsey's results with respect to the Liouville function. Here we will show another analogue of their work with respect to Ramanujan sum, leaving the investigations of analogues of results in \cite{sw19} and \cite{kms19} to the interested readers.

For any positive integers $n$ and $m$, the Ramanujan sum $c_n(m)$ to modulus $n$ is defined as
$$c_n(m):=\sum_{\smat{1\leq q\leq n\\(q,n)=1}}e^{\frac{2\pi i qm}n},$$
which was introduced by Ramanujan \cite{r1918} in 1918. 
For fixed $m$, $c_n(m)$ is multiplicative on $n$ and $c_n(1)=\mu(n)$. 

Let $K/\Q$ be a finite Galois extension and let $G := \Gal(K/\Q)$ be the Galois group. For unramified prime $p$,  set
$$\left[\frac{K/\Q}{p}\right]:=\set{\left[\frac{K/\Q}{\fkp}\right]: \fkp \subseteq\cO_K  \text{ and } \fkp|p}$$ 
where $\left[\frac{K/\Q}{\fkp}\right]$ is the Artin symbol for Frobenius map.
It is well-known that $\left[\frac{K/\Q}{p}\right]$ is a conjugacy class in $G$. Our main result in this note is the following analogue of Dawsey's result \cite{d17}.

\begin{theorem}\label{mainthm}
	Let $m\geq 1$ be an integer. Let $K$ be a finite Galois extension of $\Q$ with Galois group $G = \Gal(K/\Q)$. Then for any conjugacy class $C\subseteq G$, we have 
	\begin{equation}\label{mainthmeq}
		-\sum_{\smat{n\geq 2\\ \left[\frac{K/\Q}{p(n)}\right]=C}}\frac{c_n(m)}{n}=\frac{|C|}{|G|}.
	\end{equation}
\end{theorem}

Due to $c_n(1)=\mu(n)$, we can read Dawsey's result  in   (\ref{mainthmeq}) by letting $m=1$.  As in \cite{d17}, if $K=\Q(\zeta_k)$ where $\zeta_k$ is the $k$-th primitive unit root and $C$ is the conjugacy class of $\ell$, we get the following analogue of
Alladi's formula with respect to the Ramanujan sum $c_n(m)$.

\begin{corollary}
	Let $k\geq 1$, $\ell$ be integers and $(\ell,k)=1$. Then for any $m\geq 1$, we have
	\begin{equation}
		-\sum_{\smat{n\geq 2\\ p(n)\equiv \ell (\on{mod}k)}}\frac{c_n(m)}{n}=\frac1{\varphi(k)}.
	\end{equation}
\end{corollary}

\section{Ramanujan sum and M\"obius function}
Let  $m\geq1$ be a fixed integer. 
The Ramanujan sum $c_n(m)$ is closely related to the M\"obius function $\mu(n)$.  For instance, it is well-known (e.g., \cite{rm13}) that
\begin{equation}\label{rsmf}
	c_n(m)=\sum_{d|(n,m)}\mu\of{\frac{n}{d}}d,
\end{equation}
from which we get that $c_n(m)=\mu(n)$ if $(n,m)=1$.  See \cite{rm13}   for more properties of $c_n(m)$.
In this section, we mainly prove the analogue of the following  Alladi's theorem with respect to $c_n(m)$.

\begin{theorem}[{\cite[Theorem 6]{a77}}] \label{mainthmmuf}
	Let $P(n)$ be the largest prime divisor of  $n$. Then for any bounded function $f$ and constant $\delta$, we have
	\begin{equation}
		\sum_{n\leq x}f(P(n))\sim \delta\cdot x
	\end{equation}
	if and only if
	\begin{equation}
		-\sum_{n=2}^\infty \frac{\mu(n)f(p(n))}{n}=\delta.
	\end{equation}
\end{theorem}

\begin{theorem} \label{mainthmf}
	For any bounded function $f$ and constant $\delta$, we have
	\begin{equation}
		\sum_{n\leq x}f(P(n))\sim \delta\cdot x
	\end{equation}
	if and only if
	\begin{equation}
		-\sum_{n=2}^\infty \frac{c_n(m)f(p(n))}{n}=\delta.
	\end{equation}
\end{theorem}

\begin{proof}
	By Theorem \ref{mainthmmuf}, it suffices to prove that 
	\begin{equation}\label{diffeq}
		\sum_{2\leq n\leq x}\frac{c_n(m)-\mu(n)}{n}f(p(n))=o(1).
	\end{equation}
	
	First, by  (\ref{rsmf}) we have
	$$c_n(m)-\mu(n)=\sum_{\smat{d|(n,m)\\ d>1}}\mu\of{\frac{n}{d}}d.$$
	
	It follows that the left side of  (\ref{diffeq}) can be rewritten as follows
	\begin{align}\label{pfeq1}
		\sum_{2\leq n\leq x}\frac{c_n(m)-\mu(n)}{n}f(p(n))&=\sum_{2\leq n\leq x} \sum_{\smat{d|(n,m)\\ d>1}} \frac{\mu\of{\frac{n}{d}}}{\frac{n}{d}}f(p(n))\nonumber\\
		&=\sum_{\smat{d|m\\d>1}}\sum_{1\leq n\leq \frac{x}{d}}\frac{\mu(n)}{n}f(p(dn))
	\end{align}
	
	For the inside summation, we set $p(1)=\infty$ for convenience. Then we separate  (\ref{pfeq1}) into two parts:
	\begin{align} \label{pfeq2}
		\sum_{1\leq n\leq x}\frac{\mu(n)}{n}f(p(dn))&=\sum_{\smat{1\leq n\leq x\\ p(n)\geq p(d)}}\frac{\mu(n)}{n}f(p(d))+\sum_{\smat{1\leq n\leq x\\ p(n)< p(d)}}\frac{\mu(n)}{n}f(p(n))\nonumber\\
		&=f(p(d))\sum_{\smat{1\leq n\leq x\\ p(n)\geq p(d)}}\frac{\mu(n)}{n}+\sum_{p<p(d)}f(p)\sum_{\smat{1\leq n\leq x\\ p(n)=p}}\frac{\mu(n)}{n}\nonumber\\
		&=f(p(d))\sum_{\smat{1\leq n\leq x\\ p(n)\geq p(d)}}\frac{\mu(n)}{n}-\sum_{p<p(d)}\frac{f(p)}{p}\sum_{\smat{1\leq n\leq \frac{x}{p}\\ p(n)>p}}\frac{\mu(n)}{n}
	\end{align}
	
	Now for the summations on $\frac{\mu(n)}n$, we consider the partial sum 
	$$M(x,y):=\sum_{\smat{1\leq n\leq x\\ p(n)> y}}\mu(n)$$
	for $x,y\geq 1$. 
	By  \cite[(3.5)]{a82}, for fixed $y$, we have
	\begin{equation}
		M(x,y)=O\of{{x}\cdot{\exp(-c_1\sqrt{\log x})}},
	\end{equation}
	where $c_1$ is a positive constant depending only on $y$.
	Then by partial summation, we get that 
	\begin{equation}\label{muupbound}
		\sum_{\smat{1\leq n\leq x\\ p(n)>y}}\frac{\mu(n)}n=O\of{\exp(-c_2\sqrt{\log x})}
	\end{equation}
	for some constant $c_2>0$.
	
	Combining  (\ref{pfeq1}), (\ref{pfeq2}) and (\ref{muupbound}) together, we get that
	\begin{equation}\label{errorterm}
		\sum_{2\leq n\leq x}\frac{c_n(m)-\mu(n)}{n}f(p(n))=O\of{\exp(-c_3\sqrt{\log x})}
	\end{equation} 
	for some constant $c_3>0$.
	This gives  (\ref{diffeq}) and completes the proof of Theorem \ref{mainthmf}. 
\end{proof}

\begin{remark} By \cite[Theorem 3]{t10}, one can also conclude a weaker bound of   (\ref{muupbound}) which is sufficient to Theorem \ref{mainthmf}:
	$$\sum_{\smat{1\leq n\leq x\\ p(n)> y}}\frac{\mu(n)}n=o(1).$$
\end{remark}

\section{Proof of Theorem \ref{mainthm}}

To prove Theorem \ref{mainthm}, we need the following theorem on the density of the largest prime divisors of integers in the finite Galois extensions of $\Q$, which can be converted into the desired formula (\ref{mainthmeq}) via Theorem \ref{mainthmf}. 

\begin{theorem}[{\cite[Theorem 2]{d17}}]\label{Pndensity}
	Under the notation and assumptions of Theorem \ref{mainthm}, we have
	\begin{equation}
		\sum_{\smat{2\leq n\leq x\\ \left[\frac{K/\Q}{P(n)}\right]=C}}1=\frac{|C|}{|G|}\cdot x+O\of{x\cdot \exp(-c_4(\log x)^{\frac13})},
	\end{equation}
	where $c_4>0$ is a constant.	
\end{theorem}

\paragraph{Proof of Theorem \ref{mainthm}.}
Define an arithmetic function $f(n)$ by
$$f(n)=\left\{ \begin{matrix}
1,& \text{if } \left[\frac{K/\Q}{p}\right]=C, n=p>1;\\
0,& \text{otherwise}.
\end{matrix}\right.$$
Then $f$ is a bounded function. By Theorem \ref{Pndensity} above, 
$$\sum_{n\leq x}f(P(n))\sim \frac{|C|}{|G|} \cdot x.$$
Then by Theorem \ref{mainthmf}, we have
$$-\sum_{n=2}^\infty \frac{c_n(m)f(p(n))}{n}= \frac{|C|}{|G|},$$
which turns out to be  (\ref{mainthmeq}). This completes the proof of Theorem \ref{mainthm}.
\qed

\begin{remark}
	Combining  \cite[(10)]{d17} and  (\ref{errorterm}) together gives an error term estimate for  (\ref{mainthmeq}):
	\begin{equation}
		-\sum_{\smat{2\leq n\leq x\\ \left[\frac{K/\Q}{p(n)}\right]=C}}\frac{c_n(m)}{n}=\frac{|C|}{|G|}+O\of{\exp(-c_5(\log x)^{\frac13})},
	\end{equation}
	where $c_5>0$ is a constant.	
\end{remark}

\begin{remark}
	In general, Theorem \ref{mainthm} holds for the following generalized Ramanujan sum $c_n(m;s,g)$  which is defined as (e.g., \cite{k17})
	$$c_n(m;s,g)=\sum_{\smat{d|n\\d^s|m}}g(d)\mu\of{\frac{n}{d}}$$
	for any fixed integers $m,s\geq1$ and any arithmetic function $g$ with $g(1)=1$. When $s=1$ and $g(d)=d$, $c_n(m;s,g)=c_n(m)$ is the classical Ramanujan sum; when $g(d)=d^s$, $c_n(m;s,f)$ is the Cohen-Ramanujan sum \cite{c1949}.
\end{remark}

\begin{remark}
	Using the arguments in this note, one can also show that for integers $m,\ell,k\geq1$ and $(\ell,k)=1$,
	\begin{equation}
		-\sum_{\smat{n\geq 2\\ p(n)\equiv \ell (\on{mod}k)}}\frac{\mu(mn)}{n}=\frac{\mu(m)}{\varphi(k)}
	\end{equation}
	and that Theorem \ref{mainthm}  holds with respect to $\mu(mn)$.
\end{remark}

\section*{acknowledgements}
	The author would like to thank his advisor Professor Xiaoqing Li for her continuous support, and Liyang Yang for helpful discussions during the conference in LA.  The author would also like to thank  the anonymous referee  for a careful reading of the paper and helpful corrections and suggestions.

\end{document}